\documentclass[12pt]{article}

\usepackage{latexsym,epsfig}
\usepackage{amssymb,amsmath,amsthm}
\usepackage{graphicx,graphics}

\usepackage{color}

\title{ A Solution to an Ambarzumyan Problem on Trees}
\author{C.K.\ Law\thanks{
Department of Applied Mathematics, National Sun Yat-sen
University, Kaohsiung, Taiwan 804,  R.O.C., and National Center
for Theoretical Sciences, Taiwan. Email: law@math.nsysu.edu.tw }
\ and \ E.\ Yanagida\thanks{Department of Mathematics,
Tokyo Institute of Technology,
 Meguro-ku, Tokyo 152-8551, Japan.
Email: yanagida@math.titech.ac.jp}}

\newtheorem{theo}{Theorem}[section]
\newtheorem{lem}[theo]{Lemma}

\newtheorem{rem}[theo]{Remark}

\newcommand{\al}{\alpha}

\newcommand{\ga}{\gamma}
\newcommand{\Ga}{\Gamma}
\newcommand{\la}{\lambda}

\newcommand{\sg}{\sigma}
\newcommand{\bfv}{{\bf{v}}}
\newcommand{\bfO}{{\bf{0}}}

\begin{document}

\numberwithin{equation}{section}

\maketitle

\begin{abstract}
We consider the Neumann Sturm-Liouville problem defined on
trees such that the ratios of lengths of edges are not necessarily
rational.  It is shown that the potential function of the
Sturm-Liouville operator must be zero
 if the spectrum is equal to that for zero potential.
This extends previous results and gives an Ambarzumyan theorem for
the Neumann Sturm-Liouville problem on trees. To prove this, we
compute approximated  eigenvalues for zero potential by using
 a generalized pigeon hole argument, and make use of recursive formulas for
characteristic functions.

 \end{abstract}
\newpage

\section{Introduction}
 \hskip0.25in
In this paper we study an eigenvalue problem for a Neumann
Sturm-Liouville operator
 defined on a  metric tree (a connected graph with no cycles)
 $\Gamma=\{ V,E\}$, where $V=\{\bfv_j:\ j=0,\ldots,J\}$ is the set of all  vertices
 and $E=\{\gamma_i:\ i=1,\ldots,I \}$ is the set of all
 edges with lengths $a_i\in (0,\infty)$.
 We let $\partial \Ga$ be the set of all pendant (boundary)
 vertices.
 Any edge connected to a pendant vertex is called
 a boundary edge. For any
 internal vertex $\bfv$, we let $I(\bfv)$ be the set of all
 indices $i$ such that the edge $\ga_i$ is
 connected to $\bfv$. The degree of $\bfv$ is defined as
 $|I(\bfv)|$.
 We also choose an arbitrary internal vertex $\bfv_0$ to
 be the root.  Then we assign local coordinates to the edges such
 that for any edge $\al_i\in E$, the endpoint of $\ga_i$ further from the root
 has  local coordinates $0$,
 while the endpoint closer to the root  has
 local coordinate $a_i$. Thus all pendant vertices have
 local coordinate $0$, while the root $\bfv_0$ has local coordinate
 $a_i$ corresponding to the edge $\ga_i$ connected to it. Thus the Neumann Sturm-Liouville problem can
 be expressed as functions $y_i$'s defined on the edges $\ga_i$'s
 satisfying
 \begin{equation} \label{eq:main}
 -y_i''+q_i (x) y_i=\la y_i\qquad 0<x<a_i, \quad i=1,2,\ldots, I,
 \end{equation}
 where each $q_i$ is the potential function defined on $\ga_i$
 ($i=1,\ldots,I$), and,
 \begin{enumerate}
 \item[(A1)] $y_i'(0)=0$ whenever $\ga_i$ is a boundary edge.
 \end{enumerate}
 \hskip0.25in
  Also continuity and Kirchhoff conditions
 are imposed at each internal vertex.
 At any internal vertex $\bfv$ other than the root, there
 are some incoming edges $\ga_i$'s and one outgoing edge
 $\ga_k$. The local coordinates of $\bfv$ is $0$ with respect to
 $\ga_k$ but $a_i$ with respect to the other $\ga_i$'s.
 \begin{enumerate}
 \item[(A2)] The continuity and Kirchhoff conditions at $\bfv$
 are defined as
 $$
 y_k(0)=y_i(a_i),\qquad \mbox{ and }\qquad y_k'(0)=\sum_{i} y_i'(a_i),
 $$
 wherever $i$ is such that $i\in I(\bfv)$ but $i\neq k$.
 \end{enumerate}
 At the root $\bfv_0$, all the connecting edges are
 incoming. Hence
 \begin{enumerate}
 \item[(A3)] the continuity and Kirchhoff conditions
 are defined as: for any $i,k\in I(\bfv_0)$,
 $$
 y_k(a_k)=y_i(a_i),\qquad \mbox{ and }\qquad \sum_{i\in I(\bfv_0)} y_i'(a_i)=0,
 $$
 \end{enumerate}

The above formulation of  the Sturm-Liouville problem on $\Gamma$ is
essentially the same as in \cite{LP}. Note that the Neumann boundary
conditions (A1) can be viewed as a special case of (A2), for the
degree of a pendant vertex is $1$, hence continuity condition is
empty and right hand side of the Kirchhoff condition vanishes.
It is well-known that the above problem has a discrete spectrum. A
real number $\lambda$ is an eigenvalue of the above problem if it
has a nontrivial solution $(y_1,\ldots,y_I)$, i.e., at least one of
$y_i$'s is nontrivial. We write $Q=(q_1,\ldots, q_{I})$, and let
$\sg(Q)$ be the set of
 eigenvalues  for the vector potential function $Q$.
 In particular, $\sigma(\bfO)$ denotes the set of eigenvalues for $Q=\bfO$, i.e.,
 $q_i \equiv 0$ for all $i=1,\ldots,I$.
 \par
In the simplest case $I=1$, the above problem is reduced to a usual
Sturm-Liouville problem on a finite interval $(0,a_1)$ with the
Neumann boundary condition.  In this case, in 1929, Ambarzumyan
\cite{A29} showed that $\sigma(Q)=\sigma(0)$ implies $Q=0$ almost
everywhere.  This seems to be the first example of an inverse
problem  where the potential function can be determined uniquely by
the spectrum, without any additional data.  Later, Chern and Shen
\cite{CS97} extended the result to vectorial Sturm-Liouville
systems.  In the case of the Dirichlet boundary condition, it was
shown by  Chern et al.~\cite{CLW} that the potential function
 must be identically equal to 0 if an additional condition
\[
 \int_0^{a_1}
q_1(x)\cos ( \frac{2\pi}{a_1} x)\, dx=0
\]
is imposed.
Recently, the Ambarzumyan problem for periodic boundary conditions
was studied by Yang et al.~\cite{Yang}.   They showed that for a
vectorial Sturm-Liouville system of dimension $d$ with the periodic
boundary condition, if the eigenvalues are $(2n\pi)^2$ with
multiplicities $2d$, then the vector potential  must be $0$. Thus the
Ambarzumyan theorem, originally specified for the Neumann boundary
condition, can be generalized to several  Ambarzumyan problems
with different boundary conditions.

The aim of this paper is to study the Ambarzumyan problem for the
Neumann Sturm-Liouville problems defined on trees. In this
direction, Pivovarchik \cite{P} showed that when $\Gamma$ consists
of three edges with equal length and one triple junction, an
 analogue of Ambarzumyan theorem is valid for the  problem
(\ref{eq:main}) with (A1)$\sim$(A3).
Later, Carlson and  Pivovarchik~\cite{CP} extended the result to
any trees such that $\{a_i/L\}$ are all rational numbers, where $L$
is the total length of $\Gamma$ given by
\[
   L := \sum_{i=1}^{I} a_i.
\]
In this case, we  can find  infinitely many  eigenvalues explicitly given by
\begin{equation} \label{explicit}
  \Big( \dfrac{m m_0}{L} \pi \Big)^2  \in \sigma(0), \qquad m=0, 1,2,\ldots
\end{equation}
for some integer $m_0$, which makes the analysis much easier than a
more general case.   As for the Dirichlet problem on a star-shaped
graph, we refer to a recent paper by
 Hung et al \cite{HLS}.  See also
   \cite{IY1,IY2,K2002,K2004,LP,P,Yanagida}
 for related results on
the Sturm-Liouville problems on graphs.

In this paper we consider the
Sturm-Liouville problem on  trees  such that $\{ a_i/L \}$ are not
necessarily rational.  The following theorem is a main
result of this paper, which gives a  solution to the
Ambarzumyan problem on general trees.

\begin{theo} \label{th:Ambar}
For the Neumann Sturm-Liouville operator defined on $\Gamma$,
$\sg(Q)=\sg(\bfO)$ implies $Q=\bfO$ almost everywhere.
\end{theo}

When we deal with general trees, we encounter two kinds of
difficulty. The first one is that we must handle trees with
arbitrary number of edges. Moreover, even if the number of edges is
given, there are various trees with different topology. In order to
handle all trees, we shall derive a recursive formula for
characteristic functions whose zeros are the square roots of
eigenvalues. The second one is that we may not have explicit
eigenvalues as  (\ref{explicit}). To overcome the difficulty, we
 approximate $\{a_i/L\}$ precisely by rational numbers
at the same time by applying a (generalized) pigeon hole argument.

 In Section~2, we shall study  direct problems and derive
 recursive formulas for characteristic functions.
  In Section~3, we compute the
 expansion of characteristic functions.   In Section~4,
 we present a key lemma for the approximation of eigenvalues.
Finally Section~5 is devoted to the proof of Theorem~\ref{th:Ambar}.

 Hereafter in this paper, we shall let $\bfv_1$ be a pendant vertex which is an endpoint
of the edge $\ga_1$,  while $\bfv_2$ is another vertex at the other
end. And without loss of generality, we assume that $\ga_2$ is
connected to $\bfv_2$, and is closer to $\bfv_0$ than $\ga_1$. 
\section{Recursive formula for characteristic functions}
%
 Let $\rho
> 0$, and let $y=C_i(x;\rho)$ and $y=S_i(x;\rho)$ be (linearly
independent) solutions of
\begin{equation} \label{eq:second}
  -y''+q_i(x) y=\rho^2 y, \qquad 0<x<a_i
\end{equation}
 with the initial conditions
 \[
y(0)=1, \qquad y'(0)=0,
\]
and
\[
  y(0)=0, \qquad y'(0)=1,
\]
respectively.  We sometimes write $C_i(x;\rho)$ as $C_i(x)$
and $S_i(x;\rho)$ as $S_i(x)$  just for simplicity.
For each edge $\gamma_i$, we write  $y_i$ as
\[
y_i = A_i C_i(x) + B_i S_i(x).
  \]
Then from (A1)$\sim$(A3),
 we have a system of linear equations for the
unknowns $\{A_i\}$ and $\{B_i\}$.  Thus we may express the
coefficient matrix of the system of linear equations as
 \begin{equation}
\Phi_N(\rho)=
\begin{bmatrix}
 & C_1'(0)&S_1'(0)&0&0&  0& \cdots &0 \\
  &  C_1(a_1)&S_1(a_1) &  -C_2(0)  & - S_2(0)& 0 & \cdots & 0 &  \\
  &   C_1'(a_1)&S_1'(a_1) &  - C_2'(0) & -S_2'(0) &*&  \cdots & *  & \\
    &0& 0  &  * &* &  * &  \cdots & *  &
\\
   & \vdots& \vdots  &  \vdots&  \vdots & \vdots & &\vdots
\\
   &0  &  *&   *  & *&  \cdots & *
\end{bmatrix}, \label{eq2.05}
 \end{equation}
where the  first row corresponds to the Neumann boundary condition
 at $\bfv_1\in \Ga$, and the second and third rows
  correspond to the
continuity condition and the Kirchhoff condition at $\bfv_2$. Note
that in the first row, we include the term $S_1'(0)$ although the
corresponding coefficient $B_1=0$. We do this in order to have a
systematic form of $\Phi_N$ during reductions, as we shall see
later. We define a characteristic function by
\[
  \varphi_N (\rho) := \det \Phi_N (\rho),
\]
so that $\lambda=\rho^2  \in \sigma (Q)$
 if and only if $\varphi_N(\rho)=0$.
We note that the characteristic function depends on the orientation
of edges and how to express the coefficient matrix, but  the set  of
zeros of $\varphi_N (\rho)$ does not depend on them.


%

%


Next we introduce another eigenvalue problem by replacing
  (A1)  with the following condition:
\begin{description}
\item[\rm(A4)]
 If $\bfv_1$ is an endpoint of any boundary edge $\ga_i$,
 the solution of (\ref{eq:main})
satisfies the zero Dirichlet  boundary condition

\[
 y_i(0)=0 \ (\mbox{or }y_i(a_i))=0
\]
At other boundary vertices, the solution satisfies the homogeneous
Neumann boundary condition as in (A1).
\end{description}
Hereafter, we call (\ref{eq:main})  with (A2)$\sim$(A4)
 the Dirichlet-Neumann problem.
 If we impose the zero Dirichlet condition at
  $\alpha(\gamma_1) \in \partial \Gamma$, then $A_1=0$.
 Thus we may express the corresponding coefficient matrix by
\[
\Phi_D (\rho)=
\begin{bmatrix}
  &C_1(0)&S_1(0) &0&0&  0& \cdots &0 \\
   &C_1(a_1)& S_1(a_1)  & - C_2(0)  &-S_2(0)& 0 & \cdots & 0 &  \\
   &C_1'(a_1)& S_1'(a_1) & - C_2'(0) & -S_2'(0) &*&  \cdots & *  & \\
   &0&  0 &  * &* &  * &  \cdots & *  &
\\
    &\vdots& \vdots &  \vdots&  \vdots & \vdots & &\vdots
\\
   &0&0 &  *&   *  & *&  \cdots & *
\end{bmatrix}.
\]
Then we define a characteristic function for the Dirichlet-Neumann
problem by
\[
  \varphi_D (\rho) := \det \Phi_D (\rho).
\]
Again, the set of zeros of
$ \varphi_D (\rho)$ does not depend on the orientation of edges  and how to express
the coefficient
matrix.  Our interest will be only in zeros of the characteristic functions,
and hence the non-uniqueness of characteristic functions
 will not affect the following argument.

For a general tree, it is not easy to express explicitly the
characteristic functions $\varphi_N$ and $\varphi_D$.  Instead, we
may compute these functions recursively as follows. Let $\tilde
\Gamma$ be a subtree of $\Gamma$ obtained by removing $\gamma_1$. We
denote by $\tilde \varphi_N$ the corresponding characteristic
function of the problem with Neumann condition at $\bfv_2$ in case
$I(\bfv)=\{ 1,2\}$.  However if the degree of $\bfv_2$ is greater
than $2$,  then we take the continuity and Kirchhoff conditions at
$\bfv_2$ instead. Thus we say $\tilde \varphi_N$ the characteristic
function for a \underline{Neumann/Kirchhoff} problem on $\tilde
\Gamma$. (see
fig. 1)
\begin{figure}
 \begin{center}
 \includegraphics*[width=5.5truein,height=4truein]{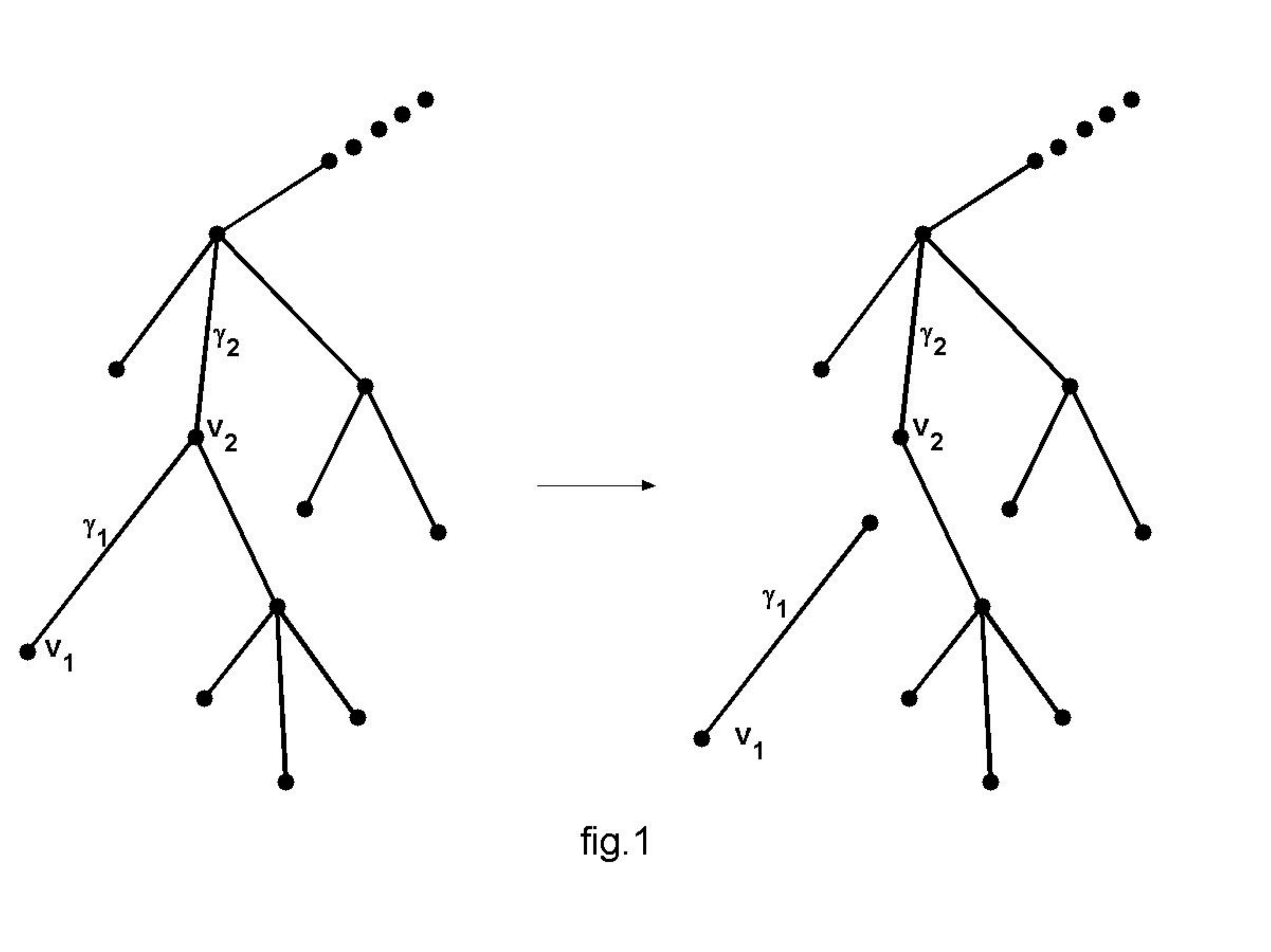}
\end{center}
\end{figure}
 \par
 If we remove $\gamma_1$ and
replace the matching conditions (A2) and (A3) by $y_j=0$ at the
vertex $\bfv_2$ for any $j\in I(\bfv_2)$, we have subtrees of
$\Gamma$ on which the Dirichlet-Neumann problems are defined. For
each subtree, a characteristic function of the Dirichlet-Neumann
problem is defined as above. We denote by $\tilde \varphi_D(\rho)$
the product of these characteristic functions.

For trees with two or more edges, we have
the following recursive formulas.


\begin{lem}  \label{le:recur}
Assume that $I \geq 2$. Then the characteristic functions have the
following properties:
\begin{description}
\item[\rm (a)] $\varphi_N(\rho) = C_1(a_1) \tilde \varphi_N(\rho)
 - C_1'(a_1) \tilde\varphi_D(\rho)$.
\item[\rm (b)] $\varphi_D(\rho) =  - S_1(a_1) \tilde \varphi_N(\rho)
+ S_1'(a_1) \tilde\varphi_D(\rho)$.
\end{description}
\end{lem}

\begin{proof}
Since $C_1'(0)=0$ and $S_1'(0)=1$,   by expansions with respect to
the first row and the first column,  we have
\[
\begin{aligned}
\det \Phi_N (\rho) &= - \det
\begin{bmatrix}
  & C_1(a_1)   & - C_2(0)  &-S_2(0)& 0 & \cdots & 0 &  \\
  &   C_1'(a_1) &-  C_2'(0) & -S_2'(0) &*&  \cdots & *  & \\
    & 0 &    * &* &  * &  \cdots & *  &
\\
   & \vdots&  \vdots&  \vdots & \vdots &  & \vdots
\\
   &0 &  *&   *  & *&  \cdots & *
\end{bmatrix}
\\ &
=   C_1(a_1) \det \tilde \Phi_N (\rho)  -  C_1'(a_1) \det \tilde
\Phi_D (\rho),
\end{aligned}
\]
where $\tilde \Phi_N$ and $\tilde \Phi_D$ are $(2I-2) \times(2I-2)$
matrices given by
\[
 \tilde \Phi_N (\rho)=
 \begin{bmatrix}
&  C_2'(0) & S_2'(0) &*&  \cdots & *  & \\
 &    * &* &  * &  \cdots & *  &
\\
&  \vdots&  \vdots & \vdots & &\vdots
\\
    &  *&   *  & *&  \cdots & *
\end{bmatrix}
\]
and
\[
 \tilde \Phi_D (\rho)=
 \begin{bmatrix}
     &  C_2(0)  &S_2(0)& 0 & \cdots & 0 &  \\
   &    * &* &  * &  \cdots & *  &
\\
  &  \vdots&  \vdots & \vdots & &\vdots
\\
    &  *&   *  & *&  \cdots & *
\end{bmatrix},
\]
respectively.  Noting that $\tilde \Phi_N$ describes conditions on
the Neumann/Kirchhoff problem for $\Gamma$  with   $\gamma_1$
removed, we have
\[
  \det \tilde \Phi_N (\rho) = \tilde \varphi_N (\rho).
\]
Similarly, since $\tilde \Phi_D$   describes conditions on
 the Dirichlet-Neumann problems for
 subtrees of $\Gamma$  obtained by removing   $\gamma_1$, we have
\[
\det \tilde \Phi_N (\rho) = \tilde \varphi_D (\rho).
\]
Thus the proof of (a) is completed.

Next, let us consider the  Dirichlet-Neumann problem. In this case,
we have the expansion, since $C_1(0)=1$ and $S_1(0)=0$,
\[
\begin{aligned}
\det \Phi_D (\rho) &=  \det
\begin{bmatrix}
 & S_1(a_1)   & - C_2(0)  &-S_2(0)& 0 & \cdots & 0 &  \\
  &   S_1'(a_1) &-  C_2'(0) & -S_2'(0) &*&  \cdots & *  & \\
     & 0 &    * &* &  * &  \cdots & *  &
\\
   & \vdots&  \vdots&  \vdots & \vdots & &\vdots
\\
   &0 &  *&   *  & *&  \cdots & *
\end{bmatrix}
\\ &
 = - S_1(a_1) \det \tilde \Phi_N (\rho) +  S_1'(a_1) \det \tilde \Phi_D (\rho).
\end{aligned}
\]
This proves (b).
\end{proof}
\begin{rem}
More general recursive formulas were obtained in a recent paper by
Law and Pivovarchik \cite{LP}. Interested readers might like to read
a spectral determinant approach to the same formulas \cite{T}.
\end{rem}

Now we compute the characteristic function of a general tree as follows.
Given a tree with two or more edges, we remove one of the
edges of $\Gamma$ and use the  recursive formulas.
Repeating this procedure, we will reach to problems on single edges.
For a single edge $\gamma_i$, we
may define its characteristic functions by
\begin{equation} \label{eq:iniphi}
  \tilde \varphi_N (\rho) := - C_i'(a_i), \qquad \tilde \varphi_D :=S_i'(a_i).
\end{equation}
Thus we can express $\varphi_N (\rho)$ and $\varphi_D (\rho)$ as
polynomials of
 $\{C_i(a_i )\}$ and  $\{S_i(a_i) \}$.

Next, we consider the zero potential $Q=\bfO$, and
 denote by $\psi_n (\rho)$ and $\psi_D (\rho)$
the corresponding characteristic functions of the Neumann/Kirchhoff
problem and the Dirichlet-Neumann problem, respectively. Similarly,
we denote by
  $\tilde\psi_N$ and $\tilde\psi_D$ be characteristic functions
  for the Neumann and
Dirichlet-Neumann problems for $\Gamma$ with   $\gamma_1$
removed.  For a tree with a single edge $\gamma_i$,
to be consistent with (\ref{eq:iniphi}),
we define its characteristic functions by
\begin{equation} \label{eq:inipsi}
\tilde\psi_N(\rho):= \sin(\rho a_i), \qquad \tilde\psi_D(\rho):=\cos(\rho a_i).
\end{equation}
For trees with two or more edges,  the characteristic functions can be computed by
using the following recursive formulas repeatedly and (\ref{eq:inipsi}).

\begin{lem}  \label{le:recur2}
 Assume $I\geq 2$.
Then the characteristic functions for $Q=\bfO$ have the following
properties:
\begin{description}
\item[\rm (a)] $\psi_N(\rho) =   \cos(\rho a_1) \tilde \psi_N(\rho)
 + \sin(\rho a_1) \tilde\psi_D(\rho)$.
\item[\rm (b)] $\psi_D(\rho) =  -  \sin(\rho a_1)
\tilde \psi_N(\rho)   +  \cos (\rho a_1) \tilde\psi_D(\rho)$.
\end{description}
\end{lem}

\begin{proof}
For $q_i(x) \equiv 0$, the solutions of (\ref{eq:second}) are given by
$C_i(x;\rho) = \cos(\rho x)$ and $S_i(x;\rho) = \sin(\rho x)$.
Then the above recursive formulas can be obtained in the same way as
 Lemma~\ref{le:recur}.
\end{proof}

\section{Expansion of characteristic functions}

 In this section we show the following
 asymptotic formulas for  $C_i(a_i;\rho)$ and $S_i(a_i;\rho)$
 as $ \rho  \to \infty$.

\begin{lem}  \label{le:rho}
As $\rho \to \infty$, one has
\[
\begin{aligned}
   C_i(a_i;\rho)&=\cos(\rho a_i)+ \rho^{-1}K_i\sin(\rho a_i)
   +o(\rho^{-1}),
\\
  C_i'(a_i;\rho)&=-\rho\sin(\rho a_i)+K_i\cos(\rho a_i) +o(1),
\\
  S_i(a_i;\rho)&=\rho^{-1}  \sin(\rho a_i)
  - \rho^{-2} K_i \cos(\rho a_i)  +o(\rho^{-2}),
\\
      S_i'(a_i;\rho)&=\cos(\rho a_i)+\rho^{-1} K_i \sin(\rho
a_i) + o(\rho^{-1}),
\end{aligned}
\]
where
\[
 K_i:=\frac{1}{2}\int_0^{a_i} q_i(x)dx .
\]
 \end{lem}

 \begin{proof}
 By the variation of constants method, it is easy to show that
 $C_i(x)$ satisfies
\[
 C_i(x) = \cos(\rho x)+\frac{1}{\rho}\int_0^x
 \sin(\rho(x-t))q_i(t)C_i(t)\, dt.
 \]
 Hence
\[
 C_i(x) = \cos(\rho x)+O(\rho^{-1})
\]
  and
\[
\begin{aligned}
  C_i(x)&= \cos(\rho x)+ \frac{1}{\rho}\int_0^x
  \sin(\rho(x-t))q_i(t) \cos(\rho t)\, dt +O(\rho^{-2}) \\
   &= \cos(\rho x)+ \frac{1}{2\rho}\int_0^x (\sin(\rho
   x)+\sin(\rho(x-2t))q_i(t)\, dt+O(\rho^{-2}) \\
   &= \cos(\rho x)+ \frac{\sin(\rho x)}{2\rho}\int_0^x q_i(t)\, dt
   +o(\rho^{-1}).
\end{aligned}
\]
   Furthermore,
\[
\begin{aligned}
   C_i'(x) &= -\rho \sin(\rho x)+\int_0^x \cos(\rho(x-t)) q_i(t)
   C_i(t)\, dt\\
        &= -\rho \sin(\rho x)+\int_0^x \cos(\rho(x-t))\cos(\rho t)
        q_i(t)\, dt+O(\rho^{-1}) \\
        &= -\rho \sin(\rho x)
            +\frac{1}{2}\int_0^x \big(\cos(\rho x)+\cos(\rho(x-2t))\big)
        q_i(t)\, dt+O(\rho^{-1}) \\
        &=-\rho \sin(\rho x)+\frac{\cos(\rho x)}{2}\int_0^x
        q_i(t)\, dt+o(1).
\end{aligned}
\]
   Evaluating at $x=a_i$, we obtain the asymptotic formulas for $C_i(a_i)$ and $C_i'(a_i)$.

   The asymptotic formulas for $S_i(a_i)$ and $S_i'(a_i)$ can  be
derived  similarly by using  
\[
   S_i(x)= \frac{\sin(\rho x)}{\rho}+\frac{1}{\rho^2}\int_0^x\sin(\rho(x-t))q_i(t)S_i(t)\,
   dt.
\]
We omit the details.
\end{proof}

\begin{lem} \label{le:expand}
The characteristic functions $\varphi_N$ and $\varphi_D$ have the
following properties:
\begin{description}
\item[\rm (a)]
$\varphi_N(\rho)=\rho\psi_N(\rho) - \Big (
 \displaystyle \sum_{i=1}^{I} K_i \Big ) \psi_D(\rho)
    +o(1)$ as $\rho\to\infty$.
\item[\rm (b)]
 $\varphi_D(\rho) = \psi_D(\rho)  +
\rho^{-1} \Big ( \displaystyle \sum_{i=1}^{I} K_i \Big )
\psi_N(\rho) +o(\rho^{-1})$ as $\rho\to\infty$.
\end{description}
\end{lem}

\begin{proof}
If $\Gamma$ consists of a single edge, then by Lemma~\ref{le:rho},
we have
\[
  \varphi_N(\rho)=- C_1'(a_1)=  \rho \sin(\rho a_1) - K_1
\cos(\rho a_1)+o(1),
\]
\[
  \varphi_D(\rho)=S_1'(a_1)= \cos(\rho a_1) +
\rho^{-1} K_1 \sin(\rho a_1) +o(\rho^{-1}),
\]
so that
\[
  \varphi_N(\rho)=\rho\psi_N(\rho) -K_1 \psi_D(\rho)
    +o(1),
\]
\[
  \varphi_D(\rho)= \psi_D(\rho) +
\rho^{-1} K_1 \psi_N(\rho) +o(\rho^{-1}).
\]
Hence (a) and (b) hold  in this case.

Suppose now that (a) and (b) hold for any subtree of
$\Gamma$.  Then by Lemmas~\ref{le:recur} and \ref{le:recur2}, we
have
\[
\begin{aligned}
   \varphi_N(\rho) &= C_1(a_1) \tilde \varphi_N(\rho)
 - C_1'(a_1) \tilde\varphi_D(\rho)
\\
 &= \Big \{ \cos(\rho a_1)
+\rho^{-1}  K_1  \sin(\rho a_1)+o(1) \Big \} \,
     \Big \{ \rho\tilde \psi_N(\rho) - \Big( \sum_{i=2} ^{I} K_i \Big)
\tilde\psi_D(\rho)+o(1)
\Big \}
\\
 & \qquad  -  \Big \{ -\rho\sin(\rho a_1)+K_1\cos(\rho a_1)
+o(1) \Big \} \, \Big \{   \tilde \psi_D(\rho) + \rho^{-1} \Big(
\sum_{i=2}^{I} K_i \Big) \tilde \psi_N(\rho) +o(\rho^{-1}) \Big \}
\\
 &=  \rho \cos(\rho a_1)\tilde\psi_N(\rho) + K_1 \sin(\rho
a_1)\tilde\psi_N(\rho) -
   \cos(\rho a_1) \Big( \sum_{i=2} ^{I} K_i \Big) \tilde\psi_D(\rho)
\\ & \qquad
   + \rho \sin(\rho a_1)\tilde \psi_D(\rho) -
K_1\cos(\rho a_1)\tilde\psi_D(\rho) + \sin(\rho a_1) \Big (
\sum_{i=2} ^{I} K_i \Big) \tilde\psi_N(\rho)  + o(1)
\\
 &= \rho \psi_N(\rho) -
K_1 \psi_D(\rho) - \Big (\sum_{i=2}^I  K_i \Big )\psi_D(\rho) +o(1),
\\
 &= \rho \psi_N(\rho) - \Big (\sum_{i=1}^I K_i \Big)\psi_D(\rho)
+o(1).
\end{aligned}
\]
Similarly,
\[
\begin{aligned}
   \varphi_D(\rho) &= -S_1(a_1) \tilde \varphi_N(\rho)
+  S_1'(a_1) \tilde\varphi_D(\rho)
\\
 &= -\Big \{ \rho^{-1} \sin(\rho a_1)
-\rho^{-2}K_1\cos(\rho a_1)+o(\rho^{-2}) \Big \} \,
     \Big \{ \rho\tilde \psi_N(\rho) -\Big( \sum_{i=2} ^{I} K_i \Big)
\tilde\psi_D(\rho) +o(1)
\Big \}
\\
 &  \qquad + \Big \{ \cos(\rho a_1)+\rho^{-1} K_1\sin(\rho a_1)
+o(1) \Big \} \, \Big \{   \tilde \psi_D(\rho) + \rho^{-1} \Big(
\sum_{i=2}^{I} K_i \Big) \tilde \psi_N(\rho) +o(\rho^{-1}) \Big \}
\\
 &=  -  \sin(\rho a_1)\tilde\psi_N(\rho)
 + \rho^{-1} \bigg \{  K_1 \cos(\rho
a_1)\tilde\psi_N(\rho) +
    \sin(\rho a_1) \Big( \sum_{i=2} ^{I} K_i \Big)
\tilde\psi_D(\rho) \bigg \} +o(\rho^{-1})
\\ &\qquad
+\cos(\rho a_1)\tilde \psi_D(\rho)
 +
\rho^{-1} \bigg \{  K_1\sin(\rho a_1)\tilde\psi_D(\rho) + \cos(\rho
a_1) \Big ( \sum_{i=2} ^{I} K_i \Big) \tilde\psi_N(\rho) \bigg \}  +
o(1)
\\
 &= \psi_D(\rho) +
\rho^{-1}\bigg \{ K_1 \psi_N(\rho) + \Big (\sum_{i=2}^I  K_i \Big
)\psi_N(\rho) \bigg \} +o(\rho^{-1})
\\
 &= \psi_D(\rho) + \rho^{-1}\Big (\sum_{i=1}^I K_i
\Big)\psi_N(\rho) +o(\rho^{-1}).
\end{aligned}
\]
Hence the assertion holds for the whole tree. Thus by induction,
the proof is complete.
\end{proof}

\section{Approximation of eigenvalues}
 \hskip0.25in
 In this section we compute approximate eigenvalues
 for $\Gamma$ with $Q=\bfO$.
We begin with the following lemma.

\begin{lem}  \label{le:Dio}
There exist infinite sequences of natural numbers $\{ m_n \} $
and $\{ k_{i,n} \} $ $(i=1,\ldots,I)$ such that $m_n \to \infty$ as
$n \to \infty$ and
\[
 \Big | \frac{a_i}{L} -\frac{k_{i,n}}{m_n} \Big | <  m_n^{-1-1/I}
\]
 for all $i=1,2, \ldots, I$ and  $n=1,2,\ldots$.
\end{lem}

\begin{proof}

If $a_i/L$ ($i=1,\ldots,I$) are all rational numbers, we can find natural numbers
$p$ and $q_i$ such that
\[
    \dfrac{a_i}{L} = \dfrac{q_i}{p}, \qquad i=1,2,\ldots,I.
\]
Then we may take
\[
  k_{i,n} = nq_i, \qquad m_n=np.
\]

Assume that not all of  $a_i/L$ ($i=1,\ldots,I$) are rational. We
consider the $I$-dimensional unit cube
\[
  \{ x=(x_1,x_2, \ldots,x_I) \in {\bf R}^I:  0 \leq x_i \leq 1
 \  \mbox{ for  } \ i=1,2, \ldots, I\},
\]
and divide it into  $n^I$ smaller cubes  with
edge length $1/n$.  We shall approximate the
vector $(a_1/L,\ldots,a_I/L)$ by a rational
point  in the unit cube as follows.
For each $p=0,1,2,\ldots, n^I$,
there exists $q_{i}\in {\bf N}$ such that
\[
 0 \leq p \frac{a_i}{L}-q_i  <  1.
\]
 Then by the pigeonhole principle, there exist vectors $U_1$ and
 $U_2$  of the form
\[
 U_j=(p_j\frac{a_1}{L} -q_{1,j},\; \ldots \; ,p_j\frac{a_I}{L}-q_{I,j}),\qquad j=1,2,
\]
that  fall into the same small cube.
Then for
 $m_n=p_2-p_1\leq n^I$ and $k_{i,n}=q_{i,2}-q_{i,1}$,  we have
 \begin{equation} \Big |
 m_n \frac{a_i}{L}- k_{i,n} \Big |  <  \frac{1}{n},\label{eq4.05}
 \end{equation}
so that
\[ \Big |
   \frac{a_i}{L}- \frac{k_{i,n}}{m_n}  \Big |  <  \frac{1}{m_n n}
\leq \frac{1}{m_n^{1+1/I}}.
\]
 The proof is complete since from (\ref{eq4.05}),
 $m_n\to\infty$ as $n\to\infty$.
\end{proof}

\bigskip

\begin{rem} \
\begin{description}
\item[\rm (a)]
Since
\[
 \Big |m_n- \sum_{i=1}^I k_{i,n} \Big | \leq
m_n  \sum_{i=1}^I \Big |   \frac{a_i}{L}- \dfrac{k_{i,n}}{m_n}  \Big |
 < \dfrac{I}{n},
\]
the equality
\begin{equation} \label{eq:mnkn}
  m_n=\sum_{i=1}^I k_{i,n}
\end{equation}
 holds  for all $n\geq I$.
  \item[\rm (b)]
  Using  $\sum_{i=1}^I (a_i/L) =1$, we may apply the
pigeonhole principle on the $(I-1)$-dimensional
hyperplane
\[
  \{ x=(x_1,x_2,\ldots,x_I ) \in {\bf R}^I : x_1+x_2+\cdots+x_I = 1 \}.
\]
Then we can improve the result
in Lemma~\ref{le:Dio} to
\[
 \Big | \frac{a_i}{L} -\frac{k_{i,n}}{m_n} \Big | =O( m_n^{-1-1/(I-1)}).
\]
  \end{description}
\end{rem}

\bigskip

 Let $\{m_n\}$ be the sequence given in Lemma~\ref{le:Dio}.  We define
 a sequence  $\{\mu_n\}$
 \[
 \mu_n:=\dfrac{2m_n\pi}L, \qquad n=1,2,\ldots,
 \]
and compute
the values of $\psi_N$, $\psi_D$, and
 their derivatives   at $\rho=\mu_n$ as follows.

\begin{lem} \label{le:psizero}
%
The characteristic functions
$\psi_N$ and $\psi_D$ have the following properties:
\begin{description}
\item[\rm (a)]
$\psi_N(\mu_n) = O(\mu_n^{-1/I})$ and
 $\displaystyle \frac{d\psi_N}{d\rho} (\mu_n)=
 L+
O(\mu_n^{-1/I})$  as $n \to \infty$.
\item[\rm (b)]
 $\psi_D(\mu_n) =
 1+ O(\mu_n^{-1/I})$ and
$\displaystyle \frac{d \psi_D}{d\rho}(\mu_n)= O(\mu_n^{-1/I}) $
  as $n \to \infty$.
\end{description}
\end{lem}
\begin{proof}
By Lemma~\ref{le:recur2} and
\[
   \mu_n a_i = \frac{m_n\pi}{L}a_i = 2 k_{i,n}  \pi + O(\mu_n^{-1/I}),
\]
we have
\[
   \psi_N(\mu_n )=  \cos(\mu_n  a_1) \tilde \psi_N(\mu_n )
+ \sin(\mu_n  a_1) \tilde\psi_D(\mu_n ) =
\tilde \psi_N(\mu_n )+ O(\mu_n^{-1/I})\tilde\psi_D(\mu_n ),
\]
and
 \[
  \psi_D(\mu_n) = - \sin(\mu_n a_1) \tilde \psi_N(\mu_n)
+\cos(\mu_n a_1) \tilde\psi_D(\mu_n)
 =
 \tilde\psi_D(\mu_n) + O(\mu_n^{-1/I})\tilde \psi_N(\mu_n).
\]
Using these equalities repeatedly on (\ref{eq:inipsi}), we obtain
\[
 \psi_N(\mu_n) = O(\mu_n^{-1/I}), \qquad  \tilde \psi_N(\mu_n)
 = O(\mu_n^{-1/I}),
\]
and
\[
 \psi_D(\mu_n) = 1 +  O(\mu_n^{-1/I}), \qquad
 \tilde \psi_D(\mu_n) = 1 +  O(\mu_n^{-1/I}).
\]

 Also, we have
 \[
 \begin{aligned}
  \frac{d\psi_N}{d\rho} (\mu_n)
&= \frac{d}{d\rho}  \Big \{ \cos(\rho a_1) \tilde \psi_N + \sin(\rho
a_1) \tilde\psi_D \Big \}  (\mu_n)
\\
&= -a_1 \sin(\mu_n  a_1) \tilde \psi_N ( \mu_n) + \cos( \mu_na_1)
\frac{d \tilde \psi_N }{d\rho} (\mu_n)
\\ & \qquad
+  a_1 \cos( \mu_n a_1) \tilde\psi_D( \mu_n)
+ \sin( \mu_n a_1) \frac{ d\tilde\psi_D }{d\rho}  (\mu_n)
\\
 &=
%
 \frac{d \tilde \psi_N }{d\rho} (\mu_n)
 + a_1
 + O(\mu_n^{-1/I})
\end{aligned}
\]
and
 \[
 \begin{aligned}
  \frac{d\psi_D}{d\rho}(\mu_n)
&= \frac{d}{d\rho} \Big \{ - \sin(\rho  a_1) \tilde \psi_N +
\cos(\rho a_1) \tilde\psi_D \Big \} (\mu_n)
\\
&=-  a_1 \cos(\mu_n a_1) \tilde \psi_N (\mu_n) - \sin(\mu_n a_1)
\frac{d\tilde \psi_N}{d\rho}  (\mu_n)
\\ & \qquad
- a_1 \sin(\mu_n a_1) \tilde\psi_D(\mu_n)
+\cos(\mu_n a_1) \frac{ d\tilde\psi_D}{d\rho}  (\mu_n)
\\
 &= \frac{ d\tilde\psi_D}{d\rho}  (\mu_n)+O(\mu_n^{-1/I}).
\end{aligned}
\]
Using these  equalities repeatedly and  (\ref{eq:inipsi}),   we obtain
\[
 \frac{d\psi_N}{d\rho} (\mu_n) = \sum_{i=1}^I a_i  + O(\mu_n^{-1/I})
\]
and
\[
  \frac{d\psi_D}{d\rho}(\mu_n) =  O(\mu_n^{-1/I}).
\]
The proof is complete.
\end{proof}

As an immediate consequence of Lemma~\ref{le:psizero}, we have
the following result concerning  the location of zeros of $\psi_N$.

\begin{lem} \label{le:eigen}
There
exists a sequence of positive numbers $\{ \rho_n \} $ such that
$\psi_N(\rho_n)=0$ for  $n=1,2,\ldots$
and  $\rho_n=\mu_n+O(\mu_n^{-1/I})$  as $n\to\infty$.
\end{lem}

\begin{proof}
 By Lemma~\ref{le:psizero}, there exists $M>0$ independent of $n$  such that
$|\psi_N(\mu_n)|\leq M\mu_n^{-1/I}$ for all large $n$.
On the other hand,
by (\ref{eq:inipsi}) and Lemma~\ref{le:recur2},
  $\psi_N (\rho) $ and  $\psi_D (\rho) $ must be some polynomials of
$\sin(\rho a_i)$ and $\cos(\rho a_i)$ ($i=1,2,\ldots, I$) with
constant coefficients.  This implies that
$ |  (d^2/d\rho^2) \psi_N  |$
is uniformly bounded in $\rho>0$.  Hence by
Lemma~\ref{le:psizero}, there exists $\delta>0$ independent of $n$ such that
\[
  \frac{d}{d\rho} \psi_N(\rho)     >  \dfrac{L}{2} \quad \mbox{ for }
  \rho \in (\mu_n - \delta, \mu_n + \delta).
\]
Therefore $\psi_N(\rho)$ must vanish at some $\rho_n$  such that
\[
  |\rho_n -\mu_n|\leq (2M/L) \mu_n^{-1/I}.
  \]
    This completes the proof.
\end{proof}

\section{Proof of Theorem~\ref{th:Ambar}}
 \hskip0.25in
In this section we complete the proof of Theorem~\ref{th:Ambar}.
By Lemma~\ref{le:expand}, $\varphi_N$  satisfies
\[
  \varphi_N(\rho)=\rho\psi_N(\rho) - \Big (
 \displaystyle \sum_{i=1}^{I} K_i \Big ) \psi_D(\rho)
    +o(1) \quad \mbox{ as } \rho \to \infty.
\]
Here, setting $\rho=\rho_n$, we have
$\varphi_N(\rho_n)=0=\psi_N(\rho_n)$ by $  \sigma(Q) =
\sigma(0)$. Also, by Lemmas~\ref{le:psizero} (b) and
\ref{le:eigen},  we have
\[
    \psi_D(\rho_n) = 1 +O(\rho_n^{-1/I}) \quad \mbox{ as } n \to \infty.
\]
Thus we obtain
\[
\Big (
 \displaystyle \sum_{i=1}^{I} K_i \Big ) \Big \{  1 +O(\rho_n^{-1/I}) \Big
\}
    +o(1) = 0.
\]
Letting $n \to \infty$, we conclude
\begin{equation}  \label{eq:Qi}
   \sum_{i=1}^{I} K_i = \frac{1}{2} \sum_{i=1}^I \int_0^{a_i} q_i(x) dx =0.
\end{equation}

On the other hand, since  $\lambda=0$ is the first eigenvalue,
it follows from the variational principle that the inequality
\[
   \frac{ \displaystyle \sum_{i=1}^I
 \int_0^{a_i}  \{ (y_i')^2+q_i(x)  (y_i)^2 \} \,  dx}
{ \displaystyle \sum_{i=1}^I  \int_0^{a_i}  (y_i)^2\, dx} \geq 0
\]
holds for any test function in $H^1(\Gamma)$.   By (\ref{eq:Qi}),
the infimum is attained by $y_i(x)=1$ for $i=1,\ldots,I$. This
implies that the constant vector $(1,\ldots,1)$ is an eigenfunction
associated with $\lambda=0$.  Then by simple substitution, we
conclude that $q_i = 0$  a.e. on $[0,a_i]$, $i=1,2,\ldots,I$.
Thus the proof is complete.  \qed


\section*{Acknowledgments}

E.\ Yanagida was supported in part by the Grant-in-Aid for
Scientific Research (A)  (No. 19204014) from the Japan Society for
the Promotion of Science. C.K.\ Law was supported in part by
National Science Council, Taiwan under grant number NSC
97-2115-M-110-005-MY2. Law would like to thank the organizers of the
Analysis on Graphs - Follow-up Workshop held at Isaac-Newton
Institute of Mathematical Sciences (July 26-30th 2010) for their
hospitality. The workshop made an impact on the final form of this
paper.


\end{document}